\documentclass[11pt]{amsart}
\usepackage{amsmath}
  \usepackage{url,graphicx}
  \usepackage{amssymb, color, pstricks-add, manfnt}
  \usepackage{amsmath, amsthm,  amsfonts}
  \usepackage[plainpages=false]{hyperref}

  \theoremstyle{plain}
  \newtheorem{theorem}{Theorem}[section]

  \newtheorem{corollary}{Corollary}[section]

   \theoremstyle{remark}

\newcommand{\Om}{\Omega}

\newcommand{\RR}{{\mathbb{R}}}
\newcommand{\Ga}{\Gamma}

\renewcommand{\phi}{\varphi}

  \numberwithin{equation}{section}
  \numberwithin{figure}{section}
\renewcommand{\baselinestretch}{1.00}
\begin{document}

\title{On the convexity theory of generating functions}

\author{Gr\'egoire Loeper}
\address{School of Mathematical Sciences, Monash University, Clayton, Vic. 3800, Australia}
\email{gregoire.loeper@monash.edu}

\author{Neil S. Trudinger}
\address{Mathematical Sciences Institute, The Australian National University,
              Canberra, ACT 0200, Australia}
%\address{School of Mathematics and Applied Statistics, University of Wollongong, Wollongong, NSW 2522, Australia}
\email{Neil.Trudinger@anu.edu.au}

\

\thanks{Research supported by  Australian Research Council Grants (DP170100929, DP180100431) }

  %  General info
\subjclass[2010]{35J60, 52A99, 78A05 }

\date{\today}

\keywords{generating functions, convexity theory }

\maketitle

\abstract {In this paper, we extend our convexity theory for $C^2$ cost functions in optimal transportation to more general generating functions, which were originally introduced by the second author to extend the framework of optimal transportation to embrace near field geometric optics. In particular we provide an alternative geometric treatment to the previous analytic approach using differential inequalities, which also gives a different derivation of the invariance of the fundamental regularity conditions under duality. We also extend our local theory to cover the strict version of these conditions for $C^2$ cost and generating functions.}
\endabstract

%\footnote { }

\baselineskip=12.8pt
\parskip=3pt
\renewcommand{\baselinestretch}{1.38}

\section{Introduction}\label{Section 1}

\vskip10pt

Generating functions were introduced by the second author in \cite{T2014} as nonlinear extensions of affine functions in Euclidean space for the purpose of extending the framework of optimal transportation to embrace near field geometric optics. Regularity and classical existence results depend on an underlying convexity theory which is of interest in its own right. Following our previous treatment in the optimal transportation case \cite{LT2020}, we develop in this paper the convexity theory under minimal smoothness assumptions on the generating function. As in \cite{LT2020}, the approach is largely geometric, somewhat shadowing that in \cite{TW2009-1}, and as a byproduct we obtain a completely different derivation of the invariance of our conditions under duality, from the complicated analytic calculation in \cite{T2014}.

A generating function can be defined on the product of two Riemannian manifolds and the real line. Here as in \cite{T2014} we will restrict attention to the Euclidean space case so that our generating functions are defined on domains $\Ga \subset \mathbb{R}^n\times \mathbb{R}^n \times \mathbb{R}$, 
whose projections,
$$I(x,y) = \{z\in\RR |\  (x,y,z)\in\Gamma\},$$
are open intervals. Assuming $g\in C^2(\Gamma)$, $g_z\ne 0$ in $\Gamma$, normalised so that $g_z\ < 0$ in $\Gamma$, and denoting 
\begin{equation}\label{U}
\mathcal{U}=\{(x,g(x,y,z),g_x(x,y,z))|\  (x,y,z)\in \Gamma\},
\end{equation} 

\noindent we then have the following two fundamental conditions from \cite{T2014},

\begin{itemize}
\item[{\bf A1}:]
For each $(x,u,p) \in \mathcal U$, there exists a unique point $(x,y,z)\in\Gamma$
satisfying
$$g(x,y,z) = u, \ \ g_x(x,y,z) = p.$$
\item[{\bf A2}:]
 $\det E \ne 0$, in $ \Gamma$, where $E$ is the $n\times n$ matrix given by
$$E = [E_{i,j}] =  g_{x,y} - (g_z)^{-1}g_{x,z}\otimes g_y.$$
\end{itemize}.

In the special case of optimal transportation,
\begin{equation}
g(x,y,z) = -c(x,y) - z,\quad   \Gamma = \mathcal D \times\RR,  \quad g_z = -1, I(x,y)  = \RR,\quad E = -c_{x,y},  
\end{equation}
\noindent where $\mathcal D$ is a domain in  $\RR^n\times\RR^n$ and $c\in C^2(\mathcal D)$ is a \emph{cost function}, satisfying conditions A1 and A2 in \cite{MTW2005}.

 By defining $Y(x,u,p)=y$ and $Z(x,u,p)=z$ in A1, the mapping $Y$ together with the dual function $Z$ are generated by the equations
\begin{equation}\label{generating equation}
g(x,Y,Z)=u, \quad g_x(x,Y,Z)=p.
\end{equation}
Since the Jacobian determinant of the mapping $(y,z)\rightarrow (g_x,g)(x,y,z)$ is $g_z\det E, \neq 0$ by A2, the functions $Y$ and $Z$ are $C^1$  smooth. By differentiating \eqref{generating equation} with respect to $p$, we also have $Y_p=E^{-1}$.

Our next fundamental condition is expressed in terms of the matrix function $A$  on $\mathcal U$, given by 

\begin{equation}\label{A}
A(x,u,p)=g_{xx}(x,Y(\cdot,u,p),Z(x,u,p)),
\end{equation}

\noindent and extends condition G3w in \cite{T2014} to non differentiable $A$.  For its formulation we use the notation $\mathcal U(x,u)$ to denote the projection $\{ p \in  \mathbb{R}^n | (x,u,p) \in \mathcal U\}$.

\begin{itemize}
\item[{\bf A3w}]:
The matrix function $A$ is  is co-dimension one convex in $\mathcal U$,  with respect to $p$,
in the sense that the function $(A\xi,\xi) (x,u,\cdot)$  is convex along line segments in $\mathcal U(x,u)$ orthogonal to $\xi$, for all $\xi \in \RR^n$,
$(x,u) \in \mathbb{R}^n \times \mathbb{R}$.
\end{itemize}

We also have a strict version of condition A3w, extending condition G3 in \cite{T2014}, namely

\begin{itemize}
\item[{\bf A3s}]:
The matrix function $A$ is  is locally uniformly co-dimension one convex in $\mathcal U$,  with respect to $p$,
in the sense that the function $(A\xi,\xi) (x,u,\cdot)$  is uniformly convex along closed line segments in $\mathcal U(x,u)$ orthogonal to $\xi$, for all $\xi \in \RR^n$, $(x,u) \in \mathbb{R}^n \times \mathbb{R}$. 
\end{itemize}
\noindent Note that here we call a function $f: I \rightarrow \mathbb{R}$ on a closed interval $I$ uniformly convex if the function, $t \rightarrow f(t) - \delta t^2$,  is convex on $I$ for some positive constant $\delta$.

When $A$ is twice differentiable in  $p$ then conditions A3w,  (A3s), can be expressed as 

\begin{equation}
(D_{p_kp_l}A_{ij}) \xi_i\xi_j\eta_k\eta_l \ge 0, (>0),
\end{equation}

\noindent in $\mathcal U$, for all  $\xi,\eta \in \mathbb{R}^n$ such that $\xi \!\cdot\! \eta = 0$. 

Historically these conditions arose from condition A3 for local regularity in optimal  transportation introduced in \cite{MTW2005}, with the weak version A3w subsequently introduced in \cite{Tru2006, TW2009} for global regularity. We will express them in the non smooth case more precisely in Section 2 and moreover show that they can also be formulated for generating functions $g \in C^1(\Gamma)$ satisfying just condition A1, corresponding to the optimal transportation case in \cite{Loe2009}, where it is also shown that condition A3w is necessary for the regularity and convexity theories. 

The strict monotonicity property of the generating function $g$ with respect to $z$, enables us to define a dual generating function $g^*$,
\begin{equation}\label{dual generating function g*}
g(x,y,g^*(x,y,u))=u,
\end{equation}
with $(x,y,u)\in \Gamma^* :=\{(x,y,g(x,y,z)) |  (x,y,z) \in \Gamma\}$, $g^*_x =-g_x/g_z$, $g^*_y=-g_y/g_z$ and $g^*_u=1/g_z$, which leads to a dual condition to A1, which is also critical for our convexity theory, namely

\begin{itemize}
\item[{\bf A1*}:]
The mapping $Q: = -g_y/g_z$ is one-to-one in $x$, for all $(y,z)$ such that $(x,y,z) \in \Gamma$.
\end{itemize}

Since the Jacobian matrix of the mapping $x\to Q(x,y,z)$ is $-E^t/g_z$ where $E^t$ is the transpose of $E$, its determinant will not vanish when condition A2 holds,  that is A2 is self dual.  We will prove the invariance of conditions A3w and A3s under duality from the local convexity theory in Section 2, which also provides an alternative proof of the case  $g \in C^4(\Gamma)$ in \cite{ T2014}, which is done there through explicit calculation of $D_{pp}A$. Note that by setting 
$$P(x,y,u) = g_x\big(x,y,g^*(x,y,u)\big),$$
\noindent we may also express condition A1 in the same form as  A1*, namely the mapping $P$ is one-to-one in $y$, for all
$(x,u)$ such that $(x,y,u)\in \Ga^*$.

\vskip10pt

\section{Local convexity}\label{Section 2}

We recall the definition from \cite{T2014} that a domain $\Omega$ is $g$-convex, (uniformly $g$-convex), with respect to 
$(y_0, z_0) \in  \mathbb{R}^n\times \mathbb{R}$,  if $(\Omega, y_0, z_0) \subset \Gamma$ and  the image $Q_0(\Omega):= Q(\cdot, y_0,z_0)(\Omega)$ is convex, (uniformly convex), in $\mathbb{R}^n$. In this section, we will examine the relationship between local $g$-convexity at boundary points of $g$-sections and condition A3w. First we note that we can write condition A3w in the form:
\begin{equation}\label{2.1}
(A\xi,\xi) (x_0,u_0, p_\theta) \le (1-\theta)(A\xi,\xi) (x_0,u_0, p_0) + \theta(A\xi,\xi) (x_0,u_0, p_1),
\end{equation}
\noindent for any $(x_0,u_0,[p_0, p_1]) \subset \mathcal U$, $p_\theta = (1-\theta)p_0 + \theta p_1$, $0\le\theta\le 1$ and $\xi.(p_1-p_0) = 0$. Here and throughout we use the notation $[p_0,p_1]$ to denote the closed straight line segment joining points $p_0$ and $p_1$ in $\mathbb{R}^n$.

Defining now 

\begin{equation}\label{2.2}
\begin{array}{rl}
& y_\theta  = Y(x_0,u_0, p_\theta), \quad z_\theta = Z(x_0, u_0, p_\theta) = g^*( x_0, y_\theta, u_0),\\ 

& h_\theta(x)  = g(x, y_\theta, z_\theta) - g(x, y_0, z_0),
\end{array}
\end{equation}

\noindent for $ x\in \Omega_0 = \mathcal U(u_0,[p_0,p_1]): = \{x  | (x, u_0, [p_0, p_1]) \in \mathcal U \}$, $\theta \in (0,1]$, we see that \eqref{2.1} can be written as

\begin{equation} \label{2.3}
(D^2 h_\theta \xi,\xi)(x_0) \le \theta (D^2h_1 \xi, \xi)(x_0),
\end{equation}

\noindent for all $\xi \in \mathbb R^n$ such that $\xi.Dh_\theta(x_0) = 0$.  This leads to the following geometric interpretation of condition A3w. Namely for  $S_\theta = \{ x \in\Omega_0 \  | \  h_\theta(x) < 0\}$, the second fundamental form $\Pi_\theta$ of $\partial S_\theta$ at $x=x_0$, with respect to an inner normal, is non-decreasing in $\theta$.  Clearly \eqref{2.3} is equivalent to $\Pi_\theta \le \Pi_1$ and the general case follows by replacing $p_1$ by $p_{\theta^\prime}$ for any $\theta^\prime\in (0,1]$. Note that $\Pi_\theta$ is well defined at $x_0$ since $\partial S_\theta \in C^2$ in some neighbourhood $\mathcal N_0$ of $x_0$. By extending the segment $[p_0,p_1]$ beyond $p_0$, we also have that $\Pi_\theta$ is bounded independently of $\theta$. From \eqref{2.3}, following the optimal transportation case in \cite {LT2020}, it also follows that condition A3w can be expressed in terms of $g$ and $g_x$ only, using just condition A1, namely:

\begin{itemize}
\item[{\bf A3v}:]
For  any $(x_0, u_0, [p_0, p_1]) \subset \mathcal U$,  we have  

$$ g(x, y_\theta, z_\theta) \le \max\{ g(x, y_0, z_0), g(x, y_1, z_1)\} + o(|x-x_0|^2),$$

\noindent for $x\in\Omega_0$,  for any  $\theta\in(0,1)$.
 
\end{itemize}

\noindent Note that  the "$o$" term in condition A3v may depend on $\theta$.

In this section, we will prove the following further equivalent characterisations of condition A3w, when condition A2 and the dual condition A1* are also satisfied. As well as providing the relationship between condition A3w  and the local g-convexity of sections $S_\theta$, the result also strengthens condition A3v by removing the "o" dependence.

\begin{theorem}\label{Th2.1}
Let $g\in C^2(\Gamma)$ be a generating function satisfying conditions A1, A2, and A1*. Then condition A3w is invariant under duality and is equivalent to the conditions:
\begin{itemize}
\item[{\bf A3w(1)}:]
For all  $(x_0, u_0, [p_0, p_1]) \subset \mathcal U$, the set $S_1$ is locally $g$-convex at $x_0$, with respect to $(y_0, z_0)$;
\item[{\bf A3w(2)}:]
For all  $(x_0, u_0, [p_0, p_1]) \subset \mathcal U$, there exists a neighbourhood $\mathcal N_0$ of $x_0$ such that

$$g(x, y_\theta, z_\theta) \le \max\{ g(x, y_0, z_0), g(x, y_1, z_1)\}$$

\noindent for all $x\in \mathcal N_0$, $\theta \in [0,1]$. 
\end{itemize}
\end{theorem}

\begin{proof}

We will prove Theorem 2.1 by proving  the implications, $A3w \Longrightarrow A3w(1) \Longrightarrow A3w(2)^*$ . The remaining assertions follow automatically, since as remarked above, $A3w(2) \Longrightarrow A3v \Longrightarrow A3w$, when $g\in C^2( \Gamma)$.

(i) $A3w \Longrightarrow A3w(1)$.

\noindent This is the main component  of the proof. First we represent the level sets $\partial S_\theta$ as graphs given by $x_n = \eta_\theta (x^\prime)$, $x^\prime=(x_1,\cdots, x_{n-1})$, near the point $x_0$, tangent at $x_0$ to the hyperplane $\{x_n = 0 \}$, with $S_\theta = \{x_n > \eta_\theta\}$ near $x_0$. Accordingly we then have $ D_i \eta (x^\prime_0) = D_i h(x_0) = 0$, $i = 1, \cdots n-1$,  $D_n h(x_0) = -\theta |p_1-p_0|$ and for $x^* =(x^\prime,x_n) \in \partial S_\theta$,  

\begin{equation} \label{2.4}
\frac{1}{\theta} Dh_\theta (x^*) \rightarrow  \frac{1}{\theta} Dh_\theta (x_0) = p_1-p_0
\end{equation}

\noindent as $x^*\rightarrow x_0$, uniformly in $\theta\in (0,1]$. To verify \eqref{2.4}, we write
\begin{equation} \label{2.5}
\begin{array}{ll}
\frac{1}{\theta} Dh_\theta (x^*)  \!\!&\!\! = \frac{1}{\theta} \{ g_x(x^*,y_\theta, z_\theta)-g_x(x^*, y_0, z_0)\}  \\
                                                  \!\!&\!\! = E(x^*, y _{\theta^\prime}, z_{\theta^\prime}) E^{-1} (x_0, y_{\theta^\prime} , z_{\theta^\prime}) (p_1-p_0),
\end{array} 
\end{equation}       
 
 \noindent for some $\theta^\prime \in (0,\theta)$, and then use the continuity of $E  =  g_{x,y} - (g_z)^{-1}g_{x,z}\otimes g_y$ with respect to $x$.  
 Defining 
 $$ y_1^* = Y(x^*, u_0^*, p_1^*), \quad z_1^* = Z(x^*, u_0^*, p_1^*)$$

\noindent where

$$ u_0^* = g(x^*,y_0,z_0), \quad p_0^*=  g_x(x^*,y_0,z_0), \quad p_1^* = p_0^* + \frac{1}{\theta} Dh_\theta (x^*),$$

\noindent it then follows that $y_1^*$ and $z_1^*$ also converge respectively to $y_1$ and $z_1$ as $x^*\rightarrow x_0$, uniformly in $\theta\in (0,1]$. 
Letting $\nu_\theta = - Dh_\theta /|Dh_\theta |$ denote the unit inner normal to  $\partial S_\theta$
and setting for $\tau^\prime \in \mathbb {R}^{n-1}$, $\tau = \tau_\theta = \tau^\prime - (\tau^\prime . \nu_\theta)\nu_\theta$, tangent to $\partial S_\theta$, we now apply condition A3w, or more precisely the monotonicity of $\Pi_\theta$ at a point $x^* = (x^\prime, x_n)\in \partial S_\theta$, near $x_0$, to obtain

$$ \frac{D_{ij} \eta_\theta (x^\prime) \tau_i \tau_j}{\sqrt{1 + |D\eta_\theta |^2}} \le \frac{\{g_{ij}(x^*,y_1^*,z_1^*) -g_{ij}(x^*,y_0,z_0)\} \tau_i \tau_j}{|p_1^*-p_0^*|}. $$ 

\noindent Sending $x^*$ to $x_0$ and using also this time the continuity of $g_{xx}$, as well as the boundedness of $\Pi_\theta$ to control the dependence on $\nu_\theta$, we then conclude for any unit vector $\tau^\prime \in  \mathbb {R}^{n-1}$, 

\begin{equation} \label{2.6}
D_{ij} \eta_\theta (x^\prime) \tau^\prime_i \tau^\prime_j \le D_{ij} \eta_1 (x^\prime) \tau^\prime_i \tau^\prime_j  + o(1)
\end{equation}

\noindent as $x^\prime \rightarrow x_0^\prime$, uniformly for $\theta\in (0,1].$   From \eqref{2.6} we now have the uniform lower bound for $\eta_1$,

\begin {equation}\label{2.7}
\eta_1(x^\prime) \ge \eta_\theta(x^\prime) + o(|x^\prime - x_0^\prime|^2
\end{equation}

\noindent as $x^\prime \rightarrow x_0^\prime$, uniformly for $\theta\in (0,1].$ Writing \eqref{2.7} in terms of $h_\theta$ we then obtain

\begin{equation} \label{2.8}
\frac{1}{\theta} h_\theta(x)  \le \max\{0,h_1(x)\} + o(|x-x_0|^2)
\end{equation}
for $x$ near $x_0$,  independently of $\theta$. Note that by exchanging $g_0$ and $g_1$ in \eqref{2.8}, we obtain a stronger version of condition A3v, without assuming condition A1*, where the "o" dependence is independent of $\theta$, namely
\begin{equation} \label {2.9}
 g(x, y_\theta, z_\theta) \le \max\{ g(x, y_0, z_0), g(x, y_1, z_1)\} + \theta (1-\theta) o(|x-x_0|^2).
 \end{equation}

\noindent Letting $\theta$ approach $0$ in \eqref{2.8}, we also obtain 

$$-g_z(x,y_0,z_0) h_0 \le \max\{h_1(x),0\} +o(|x-x_0|^2),$$

\noindent  where 

\begin{equation}\label{2.10}
 h_0 = E^{-1}(x_0,y_0,z_0) (p_1-p_0)\cdot [Q(x,y_0,z_0)-Q(x_0,y_0, z_0)]
 \end{equation}
 
 \noindent is the defining function of the $g$-hyperplane,  $S_0 = \{h_0 = 0\}$. 
 Now using condition A1*, making the coordinate transformation $x \rightarrow q= Q(x, y_0, z_0)$ and defining
 $\tilde S_1 = Q(S_1)$, $\tilde h_1(q) = h_1(x)$,  so that  $\tilde h_1$ is a defining function for $\tilde S_1$ near $q_0 = Q(x_0,y_0, z_0)$, we then obtain, using the Lipschitz continuity of $ Q ^{-1} (\cdot, y_0,z_0)$ and the positivity of $-g_z(\cdot,y_0, z_0)$,
 
$$\tilde h_1(q) \geq l(q) - o(|q-q_0|)^2$$

\noindent where $l$ is an affine function. It thus follows that the set $ \tilde S_1 = \{\tilde h_1 <0 \}$ is locally convex at $q=q_0$ and we complete the proof of assertion (i). 

(ii) $A3w(1) \Longrightarrow A3w(2)^*$

First we note that the local $g$-convexity of $S_1$ at $x_0$ means that there exists a neighbourhood $\mathcal N_0$ of $x_0$ such that $S_1\cap \mathcal N_0$ is $g$-convex with respect to $y_0,z_0$. Consequently for any point $x\in S_1\cap \mathcal N_0$, the g-segment joining $x_0$ and $x$ also lies in $S_1\cap \mathcal N_0$. Defining now $q_0 = Q(x_0, y_0, z_0)$, $q_1 = Q(x, y_0,z_0)$ and $q_\theta = (1-\theta) q_0 + \theta q_1$, we thus have 

$$ g(x_\theta, y_1, z_1) \le  g(x_\theta, y_0, z_0) : = u_\theta,$$

\noindent for  $x_\theta =  Q^{-1} ( q_\theta, y_0, z_0)$, which is equivalent to 

$$ g^*(x_\theta, y_1, u_\theta) \le z_1 = g^*( x_0, y_1, u_0).$$ 

\noindent Taking $y=y_1, x_1= x$ and exchanging $x_0$ and $x_1$,  we thus obtain for $(y_0,z_0, [q_0, q_1]) \subset \mathcal V: = \{ y, z, Q(x, y, z) | (x,y,z) \in \Gamma \}$, 

\begin{equation} \label{2.11}
g^* ( x_\theta, y, u_\theta) \le \max \{g^* ( x_0, y, u_0), g^* ( x_1, y, u_1)\}
\end{equation}
 
\noindent  for $y$ in some neighbourhood $\mathcal N_0^*$ of $y_0$ and $\theta \in [0,1])$, provided $x_1$ is sufficiently close to $x_0$. By expressing the interval $[q_0, q_1]$ as the union of sufficiently small subintervals we then conclude the dual condition A3w(2)$^*$.  

From (i) and  (ii) we then have $A3w \Longrightarrow A3w(1) \Longrightarrow A3w^* \Longrightarrow A3w (2) \Longrightarrow A3w$ so that Theorem 1.1 is completely proved. 

\end{proof}

We remark here that the proof of Theorem 2.1 is somewhat different from that of the corresponding results in the optimal transportation case in
Theorem 1.2 of  \cite{LT2020} in that it avoids the measure theoretic argument in Lemma 2.3 of \cite{LT2020}. When the third derivatives $g_{xxy}$ and $g_{xxz}$ so that the matrix function $A$ is differentiable with respect to $p$ and implication (i) follows directly from Lemma 2.4 in \cite{T2014}, in accordance with the optimal transportation case in Section 2.1 of \cite{LT2020}. Alternatively, in this case the mapping $Q$ will be twice differentiable in $x$ and we can simplify the proof of implication (i) through a $C^2$ coordinate change to express $h_0$ as an affine function in the $q$ variable so that the result then follows straight from the monotonicity of $\Pi_{\theta}$. The equivalence of A3w and A3w(2) for $C^4$ cost functions in optimal transportation goes back to \cite{Loe2009}, where it plays a fundamental role in showing the sharpness of condition A3w for regularity. 

By modification of the preceding arguments we can prove analogous equivalent versions of the strong condition A3s, including its invariance under duality. For this it is convenient to fix a subset $\mathcal U^\prime\subset \subset \mathcal U$. Then we can write condition A3s in the form

\begin{equation}\label{2.12}
(A\xi,\xi) (x_0,u_0, p_\theta) \le (1-\theta)(A\xi,\xi) (x_0,u_0, p_0) + \theta(A\xi,\xi) (x_0,u_0, p_1) - \delta \theta (1-\theta) |p_1-p_0|^2 |\xi|^2,
\end{equation}

\noindent for any $(x_0,u_0,[p_0, p_1]) \subset \mathcal U^\prime$,  $0\le\theta\le 1$ , $\xi.(p_1-p_0) = 0$ and some positive constant $\delta$, depending  on $\mathcal U^\prime$.  In place of \eqref{2.3}, we then have, for $h_{\theta,\delta} := h_\theta - \frac{\delta}{2}|p_\theta -p_0|^2 |x-x_0|^2$,

\begin{equation} \label{2.13}
\frac{1}{\theta} (D^2 h_{\theta,\delta} \xi,\xi)(x_0) \le  (D^2h_{1,\delta} \xi, \xi)(x_0) 
\end{equation}

\noindent for all $\xi \in \mathbb R^n$ such that $\xi.Dh_{\theta,\delta}(x_0) = 0$. Now, setting $S_{\theta,\delta} = \{ x \in\Omega_0 \  | \  h_{\theta,\delta}(x) < 0\}$, we can state the following strong version of Theorem 2.1. 
\begin{theorem}\label{Th2.2}
Let $g\in C^2(\Gamma)$ be a generating function satisfying conditions A1, A2, and A1*. Then condition A3s is invariant under duality and is equivalent to the conditions:
\begin{itemize}
\item[{\bf A3s(1)}:]
For all  $(x_0, u_0, [p_0, p_1]) \subset\mathcal U^\prime\subset \subset \mathcal U$, the set $S_{1,\delta}$ is locally  $g$-convex at $x_0$, with respect to $(y_0, z_0)$;
\item[{\bf A3s(2)}:]
For all  $(x_0, u_0, [p_0, p_1]) \subset \mathcal U^\prime\subset \subset \mathcal U$, there exists a neighbourhood $\mathcal N_0$ of $x_0$  
and constant $\delta_0 >0$ such that
\begin{equation}\label {2.13}
g(x, y_\theta, z_\theta) \le \max\{ g(x, y_0, z_0), g(x, y_1, z_1)\} - \delta_0 [\theta(1-\theta)|p_1-p_0| |x-x_0|]^2
\end{equation}
\noindent for all $x\in \mathcal N_0$, $\theta \in [0,1]$. 
\end{itemize}
\end{theorem}

\begin{proof}

First we may prove that A3s(2) $\Longrightarrow$ A3s by modification of the A3w case, A3v $\Longrightarrow$ A3w. Here though we should restrict the range of $\theta$, a convenient choice being $\theta = 1/2$, which would  then imply $S_{1,\delta}\cap \mathcal N_0 \subset S_{\theta,\delta}$  for
$\delta = \delta_0 /6$,  if A3s(2) holds, and hence \eqref{2.13} for $\theta = 1/2$, which  still suffices to obtain A3s in general.   

Next, the implication $A3s \Longrightarrow A3s(1)$ follows by replacing $h_\theta$ by $h_{\theta,\delta}$ in the proof of the corresponding case (i) in Theorem 2.1.

To prove $A3s(1) \Longrightarrow A3s(2)^*$ we fix a neighbourhood $\mathcal N_0$ of $x_0$ such that $S_{1,\delta}\cap \mathcal N_0$ is $g$-convex with respect to $y_0,z_0$ and for $x = x_1 \in S_{1,\delta}\cap \mathcal N_0$, we define  $q_\theta$, $x_\theta$ and $u_\theta$ as in the proof of case (ii) of Theorem 2.1. Then we have

$$ g(x_\theta, y_1, z_1) + \frac{\delta}{2}|p_1 -p_0|^2 |x_\theta -x_0|^2 \le  g(x_\theta, y_0, z_0) = u_\theta,$$

\noindent so that by the mean value theorem,

$$ g^*(x_\theta, y_1, u_\theta) \le z_1 + \frac{\delta}{2}|p_1 -p_0|^2 |x_\theta -x_0|^2 g^*_u(x_\theta, y_1, u^*),$$ 

\noindent for some $u^*$, satisfying

$$ g(x_\theta, y_1, z_1) \le u^* \le g(x_\theta, y_1, z_1) + \frac{\delta}{2}|p_1 -p_0|^2 |x_\theta -x_0|^2. $$

\noindent Consequently, since  $g^*_u < 0$, we obtain for $y = y_1$ sufficiently close to $y_0$, 

$$ g^*(x_\theta, y_1, u_\theta) \le g^*( x_0, y_1, u_0) - \kappa_0 \delta \theta^2 |q_1-q_0|^2 |y-y_0|^2$$

\noindent for some positive constant $\kappa_0$, depending on $g$ and $\mathcal U^\prime$. Exchanging $x_0$ and $x_1$, and consequently replacing $\theta$ by $1-\theta$, we then obtain, in place of \eqref{2.11},

\begin{equation}
g^* ( x_\theta, y, u_\theta) \le \max \{g^* ( x_0, y, u_0), g^* ( x_1, y, u_1)\} - \delta^*_0[\theta(1-\theta)|q_1-q_0| |y-y_0|]^2
\end{equation}

\noindent for some constant $\delta^*_0$,  for $y$ in some neighbourhood $\mathcal N_0^*$ of $y_0$ and $\theta \in [0,1])$, provided $x_1$ is sufficiently close to $x_0$, and hence infer the dual condition A3s(2)$^*$. 

Corresponding to the proof of Theorem 2.1, which also can be viewed as the limit case $\delta = 0$, we then have $A3s \Longrightarrow A3s(1) \Longrightarrow A3s^* \Longrightarrow A3s (2) \Longrightarrow A3s$, which completes the proof of Theorem 2.2.

\end{proof}

The case $A3s \Longrightarrow A3s(2)$ in Theorem 2.2  extends to $C^2$ generating functions the corresponding result in Lemma 4.5 in \cite{T2020}, which is proved there, similarly to the basic convexity results under A3w, by using the differential inequality approach. We may also express the condition A3s(1) in terms of a local uniform $g$-convexity of $S_1$.Note also that, without assuming the dual condition A1*, we obtain from the proof of the implication $A3s \Longrightarrow A3s(1)$, (in particular from
the estimate \eqref{2.8} applied to $h_{\theta,\delta}$), that condition A3s is equivalent to a strong form of condition A3v, corresponding to \eqref{2.9}, namely

\begin{equation}\label{2.17}
g(x, y_\theta, z_\theta) \le \max\{ g(x, y_0, z_0), g(x, y_1, z_1)\} - \theta(1-\theta) [\delta |p_1-p_0|^2 |x-x_0|^2 
+ o(|x-x_o|^2)],
\end{equation}

\noindent  for all $(x_0, u_0, [p_0, p_1]) \subset\mathcal U^\prime\subset \subset \mathcal U$, $x \in \Omega_0$ and some positive constant $\delta$, with the "o" dependence independent of $\theta$.

\vskip10pt

 \section{Global convexity}\label{Section 3}
 
 In this section we deduce from Theorem 1.1, fundamental properties of $g$-convex functions when $g$ is only assumed $C^2$. First we recall from \cite{T2014} that a  function $u\in C^0(\Om)$ is called $g$-\emph{convex} in  $\Om$,
 if for each $x_0\in\Om$, there exists $(y_0,z_0) \in\mathbb{R}^n\times \mathbb{R}$ such that  $(\Omega,y_0,z_0) \subset \Gamma$ and 
\begin{align}\label{3.1}
u(x_0) &= g(x_0,y_0,z_0), \\
 u(x) &\ge g(x,y_0,z_0)\notag
  \end{align}
\noindent for all $x\in\Om$. If $u$ is differentiable at $x_0$, then $y_0 = Tu (x_0): = 
Y(x_0,u(x_0),Du(x_0))$,
while if $u$ is twice differentiable at $x_0$, then
\begin{equation}\label{3.2}
 D^2u(x_0) \ge  g_{xx}(x_0,y_0,z_0) = A(\cdot,u,Du)(x_0)
\end{equation}
\noindent We  also refer to functions of the form $g(\cdot,y_0,z_0)$ as \emph{$g$-affine} and as
a \emph{$g$-support} at $x_0$ in $\Omega$ if \eqref{3.1} is satisfied. Note also that the $g$-convexity of a function
 $u$ in $\Omega$ implies its local semi-convexity. 
 
If  $u$ is a $g$-convex function on $\Omega$, extending the differentiable case, we define  the \emph{$g$-normal} mapping
of $u$ at $x_0\in\Om$ to be the set: 

$$Tu(x_0) = \big\{ y_0\in \RR^n \mid 
  \Omega \subset \Gamma_{y_0,z_0} \text{ and } \quad u(x) \ge g(x,y_0,z_0)\text{ for all } x\in\Om\big\},$$
 
\noindent where $z_0 = g^*(x_0,y_0,u_0), u_0 =u(x_0)$. Note that if $u =g(\cdot,y,z)$ is  $g$-affine, then $Tu = y$, while in general

$$ Tu (x_0) \subseteq \Sigma_0 = \Sigma_u (x_0) :=Y(x_0,u(x_0),\partial u(x_0)), $$

\noindent where $\partial u$ denotes the sub differential of $u$, provided the extended one jet, 
$J_1[u](x_0) = [x_0,u(x_0), \partial u(x_0)]\subset \mathcal U$

Next if $g_0 = g(\cdot,y_0,z_0)$ is a $g$-affine function, we define the \emph{section} of  a $g$-convex function $u$ with respect to $g_0$ by
 
$$ S(u,g_0) =  \big\{ x\in\Om \mid u(x) <  g(x,y_0,z_0) \big\} $$ 

\noindent We can also  have a notion of  closed sections, (as used in \cite {LT2020}) given by 

$$ \tilde S(u,g_0) =  \big\{ x\in\Om \mid u(x) \le  g(x,y_0,z_0) \big\}, $$ 

\noindent which includes, as a special case, the contact set of  $u$ with respect to $g_0$,

$$  S_0(u,g_0) =  \tilde S(u,g_0) = \big\{ x\in\Om \mid u(x) = g(x,y_0,z_0) \big\}, $$ 

\noindent when $g_0$ is a $g$-support of $u$.

Our approach here will be based on the following global extension of Theorem 2.1(i), which extends result (ii) in Theorem 1.2 in \cite{LT2020} to the generating function case. 

 \begin{theorem} 
Assume $g$ satisfies A1,A2,A1* and A3w, $u\in C^0(\Omega)$ is $g$-convex and $g_0 = g(\cdot, y_0,z_0)$ is  $g$-affine in a domain $\Omega$. Assume also:
\vspace{2mm}

(i) $\Omega$ is $g$-convex with respect to $(y_0, z_0);$
\vspace{2mm}

(ii) $(\cdot, g_0, [Dg_0,g_x(\cdot,y,g^*(\cdot,y,g_0))])(\Omega) \subset \mathcal U$ for all $y\in Tu(\Omega)$. %$g$-affine supports $g_1$ to $u.$
\vspace{2mm}

\noindent Then the sections $S =  S(u,g_0)$ and $\tilde S= \tilde S(u,g_0)$ are also $g$-convex with respect with respect to $(y_0, z_0)$.
\end{theorem}

\begin{proof}
To prove Theorem 3.1 we  follow the corresponding argument in the optimal transportation case \cite{LT2020}, modified in accordance with  the proof of Lemma 2.3 in \cite{T2020}. First we replace $\Omega$  by a $C^1$ subdomain $\Omega^\prime\subset\subset \Omega$, which is also $g$-convex with respect to $(y_0, z_0)$ and consider the special case, $u = g_1$ for some fixed $g$-affine function $g_1 = g (\cdot, y_1,z_1)$. From condition (ii) we then have $(x,y_1,z) \in \Gamma$ for all $x\in \Omega^\prime$, $z\ge z_1$ and $g(x,y_1,z) \ge  g_0(x) -  \epsilon$, for some constant $\epsilon>0$. Now suppose that the set $S_1 = S(g_1,g_0)$ has two disjoint components. By increasing $z_1$  and writing $g_{1, \delta} = g(\cdot, y_1, z_1+\delta)$, $S_{1,\delta}= S(g_{1, \delta}, g_0)$ for $\delta \ge 0$, we then obtain, from the $g$-convexity of $\Omega^\prime$, that the section $S_{1,\delta}$  has two distinct components for some $\delta \ge 0$, touching in $\bar\Omega$. From the local convexity, A3w(1) in Theorem 2.1, this can only happen at a point $\hat x \in \partial\Omega^\prime$.  For sufficiently small $\rho$, we then have that $g_{1, \delta} < g_0$ in 
$B_\rho\cap \partial\Omega^\prime - \{\hat x\}$ while  $g_{1, \delta}(x)  > g_0 (x)$ for $x = \hat x +t\nu$, $0 <t < \rho$, where  $\nu$ denotes the unit inner normal to $\partial \Omega^\prime$ at $\hat x$, which contradicts  the local $g$-convexity of $\Omega^\prime$ at $\hat x$. Consequently $S_1$ is connected and since it is  locally $g$-convex with respect to $(y_0,z_0)$, it is also globally $g$-convex with respect to $(y_0,z_0)$. Replacing $S_1$ by  $S_{1,\delta}$ and letting $\delta \rightarrow 0$, we also obtain the $g$-convexity of $\tilde S_1=\tilde S(g_1,g_0)$.

\noindent For the general case we write for $u$, $g$-convex in $\Omega$, 
$$\tilde S (u,g_0) = \cap \{\tilde S(g_1, g_0) \ | \ g_1 \text{is a $g$-support to} \ u \}$$
which gives the $g$-convexity of $\tilde S$ and consequently $S$ in general.
\end{proof}
 
 For further results  we will use the sub-convexity notion introduced in Section 2 of \cite{T2020} so that, for example, condition (ii) in Theorem  3.1 can be written equivalently as $\{y_0,y\}$ is sub $g^*$-convex with respect to $g_0$ on $\Omega$ for all $y\in Tu(\Omega)$ or that the $g^*$-segment, with respect to $(x,g_0(x))$, joining $y_0$ and $y$ is well defined for all $x\in \Omega, y\in Tu(\Omega)$. Recall also that the $g$-transform of a $g$-convex function $u$,  on a domain $\Omega$, is defined by 
\begin{equation} \label{3.4}
v(y) = u^*_g(y) = \text{sup}_{\Om} \ g^*(\cdot,y,u)
\end{equation}
\noindent for $y\in Tu(\Omega)$, so that from \eqref{3.1}, $v(y_0) = z_0$, if $g_0= g(\cdot,y_0,z_0)$   is a $g$-support to $u$.
 
  From Theorem 3.1 and the invariance of A3w under duality, we then have the following global extension  of Theorem 2.1(ii).

 \begin{corollary} 
Assume $g$ satisfies A1,A2,A1* and A3w and $u\in C^0(\Omega)$ is $g$-convex in a domain $\Omega$. Then, if for some $x_0 \in \Omega$ and all $x \in \Omega$,  the pair $\{x_0,x\}$ is sub $g$-convex with respect to $v = u^*_g$ on $\Sigma_0$, we have $Tu (x_0) = \Sigma_0$ is $g^*$-convex, with respect to $x_0$ and $u_0= u(x_0)$. 
\end{corollary}

Note that by the semi-convexity of $u$,  $P( x_0, u_0, \Sigma_0)$ is the convex hull of  $P(x_0, u_0,Tu(x_0))$ so Corollary 3.1  follows from the $g^*$- convexity of $Tu(x_0)$, which in turn follows directly using duality with domain $\Omega^*$ a neighbourhood of $\Sigma_0$, which is also $g^*$- convex with respect to $x_0,u_0$.  But we may also proceed slightly differently as in \cite{T2020} by proving first a special case when $u$ is replaced by $\max\{ g_0, g_1\}$ where $g_0$ and $g_1$ are two g-affine functions satisfying $g_1(x_0) = g_0(x_0) = u_0$. Then using the notation in Theorem 3.2, if $\{x_0, x\}$ is sub $g$-convex with respect to $(y_\theta, z_\theta)$, for all $\theta\in (0,1)$, we have the inequality,
\begin{equation}
g(x, y_\theta, z_\theta) \le \max\{ g_0(x), g_1(x)\} 
\end{equation}
which is the global version of A3w(2) in Theorem 2.1.

Finally we consider the global $g$-convexity of locally $g$-convex functions thereby  extending Lemmas 2.1 in \cite{T2014, T2020} to the non-smooth case and part (iv) of Theorem 1.2 in \cite{LT2020} to the generating function case. Here we will define a function $u\in C^0(\Omega)$ to be \emph{locally} $g$-\emph{convex} in $\Omega$, at any point $x_0\in \Omega$, $u$ has a $g$-support in some  neighbourhood $\mathcal N_0$ of $x_0$. 

%Note that this coincides with our definition for 
%$u\in C^2$ in \cite{T2014, T2020} when the matrix function $A$ is Lipschitz continuous with respect to the gradient variables or when $u$ is elliptic, that is inequality \eqref{3.2} is strict in $\Omega$. 
%We can also define $Tu(x_0)$ by replacing $\Omega$ by $\mathcal N_0$ so that from Corollary 3.1 or Theorem 2.1(ii),  $Tu(x_0) = \Sigma _0$, as in the global case. 

\begin{theorem}
Assume $g$ satisfies A1,A2,A1* and A3w and $u\in C^0(\bar \Omega)$ is locally $g$-convex in a domain $\Omega$.  Assume also:
\vspace{2mm}

(i) $\Omega$ is $g$-convex with respect to $(y,z)$ for all $y\in \Sigma_u(\Omega)$, $z \in g^*(\cdot , y, u) (\bar\Omega)$
\vspace{2mm}

(ii) $\Sigma_u(\Omega)$ is sub $g^*$-convex with respect to $u$ on $\Omega$.
\vspace{2mm}

Then $u$ is $g$-convex in $\Omega$. 
\end{theorem}

Note that in \cite{T2014, T2020}, we have defined local $g$-convexity for a $C^2$ function $u$ by the degenerate ellipticity condition \eqref{3.2}. Clearly
if $u$ is elliptic in $\Omega$, that is inequality \eqref{3.2} is strict in $\Omega$, then $u$ is locally $g$-convex as above. From this it follows by approximation, $u \rightarrow u +\epsilon (x-x_0)^2$, for small $\epsilon >0$ and Theorem 3.2, that our definitions are equivalent if $A$ is Lipschitz continuous with respect to the $u$ and $p$ variables. 

Corresponding to  \cite{LT2020}, the proof of Theorem 3.2 is just a modification of that of Theorem 3.1.

Finally we remark that from Theorem 2.2, by adapting the approach in \cite{Loe2009}, we can obtain the $C^1$ and $C^{1,\alpha}$ regularity of generalized solutions of the second boundary value for generated Jacobian equations for $C^2$ generating functions satisfying  A1,A2, A1* and A3s under appropriate integrability or boundedness conditions on the initial and target densities, $f$ and $f^*$, and convexity conditions on the initial and target domains, $\Omega$ and $\Omega^*$. In particular, the corresponding regularity results in the optimal transportation case, in Theorems 3.4 and 3.7  of \cite{Loe2009}, may be extended to $C^2$ cost functions  while their extensions to generated Jacobian equations, proved recently in Theorem 2.14 of \cite{Je}, may be extended to $C^2$ generating functions. In fact these extensions do not need the full strength of the implication $A3s\longrightarrow A3s(2)$ in Theorem 2.2 and the estimate \eqref{2.15},  which already is a refinement of Proposition 5.1 in \cite{Loe2009} and Lemma 3.3 in \cite{Je}, is sufficient. For this we also need the characterisation of the $g$-normal mapping in Corollary 3.1. 

For a formulation of the generalized second boundary value problem for generated Jacobian equations we may refer, for example, to Section 4 in \cite{T2014} or Section 3 in \cite{T2020}.

\end{document}